\documentclass[11pt]{article}
\usepackage{amssymb}
\usepackage{amsthm}
\usepackage{amsmath}
\usepackage{fullpage,hyperref}
\usepackage{graphicx,url,amsmath,amsthm,amsfonts,amssymb,subfigure,bbm}

\newtheorem{theorem}{Theorem}[section]
\newtheorem{lemma}[theorem]{Lemma}

\newtheorem{corollary}[theorem]{Corollary}

\newtheorem{definition}[theorem]{Definition}

\newtheorem{conjecture}{Conjecture}

\newcommand{\jnote}[1]{}

\newcommand{\E}{{\mathbb E}}

\newcommand{\dem}{\mathrm{dem}}

\newcommand{\remove}[1]{}
\newcommand{\1}{\mathbf{1}}

\newcommand{\len}{\mathsf{len}}

\newcommand{\rank}{\mathsf{rank}}

\newcommand{\cP}{\mathcal{P}}

\newcommand{\fint}{\mathsf{int}}

\newcommand{\embeds}{\rightsquigarrow}
\newcommand{\edgerank}{\mathsf{edge\mbox{-}rank}}

\newcommand{\pr}{\mathbb{P}}

\begin{document}

\title{{\bf Pathwidth, trees, and random embeddings}\vphantom{\footnote{A portion of the results in this paper were announced at the 41st Annual Symposium
on the Theory of Computing \cite{LS-STOC09}.}}}

\author{James R. Lee\thanks{Computer Science \& Engineering, University of Washington. Research partially supported by NSF grant CCF-0644037 and a
Sloan Research Fellowship.  E-mail: {\tt jrl@cs.washington.edu}.} \and Anastasios Sidiropoulos\thanks{Toyota Technological Institute at Chicago.  E-mail: {\tt tasos@ttic.edu}.}}

\date{}

\maketitle

\begin{abstract}
We prove that, for every integer $k \geq 1$, every shortest-path metric on a graph of pathwidth $k$ embeds
into a distribution over random trees with distortion at most $c=c(k)$, independent of the graph size.
%This result is new even for $k=3$.
A well-known conjecture of Gupta, Newman, Rabinovich, and Sinclair \cite{GNRS99} states that for every
minor-closed family of graphs $\mathcal F$, there is a constant $c(\mathcal F)$ such that
the multi-commodity max-flow/min-cut gap for every flow instance on a graph from $\mathcal F$
is at most $c(\mathcal F)$.  The preceding embedding theorem is used to prove this
conjecture whenever the family $\mathcal F$ does not contain all trees.
\end{abstract}

\section{Introduction}
\label{sec:intro}

We view an undirected graph $G=(V,E)$ as a topological
template that supports a number of different geometries.
Such a geometry is specified by a non-negative
length function $\len : E \to [0,\infty)$ on edges,
which induces a shortest-path pseudometric $d_{\len}$ on $V$,
with $$d_{\len}(u,v) = \textrm{length of the shortest path between $u$ and $v$ in $G$},$$
where a pseudometric might have $d_{\len}(u,v)=0$ for some pairs $u,v \in V$ with $u \neq v$.
A pseudometric $d$ is {\em supported on $G$} if $d=d_{\len}$ for some such $\len : E \to [0,\infty)$.
From this point of view, we are interested in properties
which hold simultaneously for all geometries supported on $G$,
or even for all geometries supported on a family of graphs $\mathcal F$.
In what follows, we will deal exclusively with {\em finite} graphs and families
of finite graphs unless explicitly stated otherwise.

In the seminal works of Linial-London-Rabinovich \cite{LLR95} and
Aumann-Rabani \cite{AR98}, and later Gupta-Newman-Rabinovich-Sinclair \cite{GNRS99},
the geometry of graphs is related to the classical
study of the relationship between flows and cuts.

\medskip
\noindent
{\bf Multi-commodity flows and $L_1$ embeddings.}
For a metric space $(X,d)$, we use $c_1(X,d)$ to denote
the {\em $L_1$ distortion of $(X,d)$}, i.e.
the infimum over all numbers $D$ such that
$X$ admits an embedding $f : X \to L_1$ with
$$
d(x,y) \leq \|f(x)-f(y)\|_1 \leq D \cdot d(x,y)
$$
for all $x,y \in X$.  Here, we have $L_1 = L_1([0,1])$,
which can be replaced by the sequence space $\ell_1$ when $X$ is finite.

Corresponding to the preceding discussion, for a graph $G=(V,E)$ we write
$c_1(G) = \sup c_1(V, d)$ where $d$ ranges over all metrics supported on $G$.
For a family $\mathcal F$ of graphs, we write $c_1(\mathcal F) = \sup_{G \in \mathcal F} c_1(G)$.
Thus for a family $\mathcal F$ of finite graphs, $c_1(\mathcal F) \leq D$ if and only if
every geometry supported on a graph in $\mathcal F$ embeds into $L_1$
with distortion at most $D$.

A {\em multi-commodity flow instance in $G$} is specified
by a pair of non-negative mappings $\mathrm{cap} : E \to \mathbb R$ and $\mathrm{dem} : V \times V \to \mathbb R$.
We write $\mathsf{maxflow}(G; \mathrm{cap}, \mathrm{dem})$ for the value of the {\em maximum concurrent flow} in this instance,
which is the maximal value $\varepsilon$
such that a flow of value $\varepsilon \cdot \dem(u,v)$ can be simultaneously routed between every pair $u,v \in V$
while not violating the given edge capacities.

A natural upper bound on $\mathsf{maxflow}(G; \mathrm{cap}, \mathrm{dem})$ is given by the {\em sparsity}
of any cut $S \subseteq V$:
\begin{equation}\label{eq:sparse}
\Phi(S;\mathrm{cap},\mathrm{dem}) = \frac{\sum_{uv \in E} \mathrm{cap}(u,v) |\1_S(u)-\1_S(v)|}{\sum_{u,v \in V} \mathrm{dem}(u,v) |\1_S(u)-\1_S(v)|},
\end{equation}
where $\1_S : V \to \{0,1\}$ is the indicator function for membership in $S$.
In the case where $\dem(u,v) > 0$ for exactly one pair $u,v$, also known as
\emph{single-commodity flow} \cite{FF56}, minimizing the upper bound \eqref{eq:sparse} over all cuts $S\subseteq V$
computes the minimum $u$-$v$ cut in $G$, and the max-flow/min-cut theorem states that this upper bound
is achieved by the corresponding maximum flow.

In general, we write
$$\mathsf{gap}(G)
= \sup_{\mathrm{cap},\mathrm{dem}} \frac{\min_{S \subseteq V} \Phi(S;\mathrm{cap},\mathrm{dem})}{\mathsf{maxflow}(G;\mathrm{cap},\mathrm{dem})}\,.
$$
for the maximum ratio between the best upper bound given by \eqref{eq:sparse}
and the value of the flow, over all multi-commodity flow instances on $G$.
This is the {\em multi-commodity max-flow/min-cut gap for $G$}.  Now we can state
the fundamental relationship between the geometry of graphs and the flows they support:

\begin{theorem}[\cite{LLR95, GNRS99}]\label{thm:gap}
For every graph $G$, $c_1(G) = \mathsf{gap}(G)$.
\end{theorem}

In the general Sparsest Cut problem (also known as Sparsest Cut with general demands), given $G$, $\mathrm{cap}$, and $\mathrm{dem}$, we want to find a cut in $G$ of minimum sparsity.
Combined with the techniques of \cite{LR99, LLR95}, Theorem \ref{thm:gap} implies that
there exists a $c_1(G)$-approximation for the general Sparsest Cut problem on a graph $G$.
Motivated by this connection, Gupta, Newman, Rabinovich, and Sinclair sought to
{\em characterize} the graph families $\mathcal F$ such that $c_1(\mathcal F) < \infty$,
and they posed the following conjecture.
We will say that a family of graphs $\mathcal F$ {\em forbids some minor}
if there exists a graph $G$ that is not a minor of any graph in $\mathcal F$.

\begin{conjecture}[\cite{GNRS99}]
\label{conj:gnrs}
For every family of finite graphs $\mathcal F$, one has $c_1(\mathcal F) < \infty$
if and only if $\mathcal F$ forbids some minor.
\end{conjecture}

We refer to Section \ref{sec:prelims} for a review of graph minors.
Progress on the preceding conjecture has been limited.
Classical work of Okamura and Seymour \cite{OS81} implies that
$c_1(\mathsf{Outerplanar}) = 1$, where $\mathsf{Outerplanar}$
denotes the class of outerplanar graphs (planar graphs where
all vertices lie on a single face).  Gupta, Newman, Rabinovich,
and Sinclair \cite{GNRS99} proved that $c_1(\mathsf{Treewidth}(2)) = O(1)$,
where $\mathsf{Treewidth}(k)$ denotes the family of all graphs
of treewidth at most $k$ (see, e.g. \cite{DiestelBook} for a discussion of treewidth, or
Section \ref{sec:prelims} for the relevant definitions).
This was improved to $c_1(\mathsf{Treewidth}(2))=2$ in \cite{LR10,CJLV08}.
Finally, in \cite{CGNRS06},  it is shown that
$c_1(\mathsf{Outerplanar}(k)) < 2^{O(k)}$ for all $k \in \mathbb N$,
where $\mathsf{Outerplanar}(k)$ denotes the class of $k$-outerplanar graphs.
We remark that a strengthening of Conjecture \ref{conj:gnrs}, regarding integer multi-commodity flows, has been investigated by Chekuri, Shepherd, and Weibel \cite{CSW10}.
The present paper is devoted to proving the following special case of
Conjecture \ref{conj:gnrs}.

\begin{theorem}\label{thm:pw}
Every minor-closed family of finite graphs $\mathcal F$ which does not contain
every possible tree satisfies $c_1(\mathcal F) < \infty$.  Equivalently,
the multi-commodity max-flow/min-cut
gap for $\mathcal F$ is uniformly bounded, i.e. $\mathsf{gap}(\mathcal F) < \infty$, whenever $\mathcal F$ has bounded pathwidth.
\end{theorem}

We remark that Theorem \ref{thm:pw} implies a polynomial-time $O(1)$-approximation algorithm for the general Sparsest Cut problem
on graphs of bounded pathwidth.
Recently, an $O(1)$-approximation algorithm for graphs of bounded treewidth has been obtained by Chlamtac, Krauthgamer, and Raghavendra \cite{CKR10}.
We also note that \cite{CKR10} uses a different approach, and does not establish an analogous bound on the multi-commodity max-flow/min-cut gap for graphs of bounded treewidth (which remains an important open problem).

\subsection{Simplifying the topology with random embeddings}

A basic question is whether one can embed a graph metric $G$ into a
graph metric $H$ with a simpler topology (for example,
perhaps $G$ is planar and $H$ is a tree), where the embedding
is required to have small distortion, i.e. such that every
pairwise distance changes by only a bounded amount.
The viability of this approach as a general method
was ruled out by Rabinovich and Raz \cite{RR98}.  For instance,
$\Omega(n)$ distortion is required to embed an $n$-cycle into a tree.
In general (see \cite{CG04}), if all metrics supported on a
subdivision of some graph $G$ can be embedded with distortion $O(1)$ into metrics
supported on
a family $\mathcal F$, then $G$ is a minor of some graph in $\mathcal F$,
implying that we have not obtained a reduction in topological complexity.

On the other hand, a classical example attributed to Karp \cite{Karp89} shows
that random reductions might still be effective:  If one removes
a uniformly random edge from the $n$-cycle, this gives an embedding
into a random tree which has distortion at most 2 ``in expectation.''
More formally, if $(X,d)$ is any finite metric space, and $\mathcal Y$
is a family of finite metric spaces, we say that {\em $(X,d)$ admits
a stochastic $D$-embedding into $\mathcal Y$} if there exists
a randomly chosen metric space $(Y,d_Y) \in \mathcal Y$ and a randomly chosen
mapping $F : X \to Y$ such that the following
two properties hold.
\begin{description}
\item[Non-contracting.] With probability one, for every $x,y \in X$, we have $d_Y(F(x),F(y)) \geq d(x,y)$.
\item[Low-expansion.] For every $x,y \in X$,
$$
\mathbb E\left[\vphantom{\bigoplus} d_Y(F(x),F(y))\right] \leq D \cdot d(x,y).
$$
\end{description}

For two graph families $\mathcal F$ and $\mathcal G$, we write $\mathcal F \rightsquigarrow \mathcal G$
if there exists a $D < \infty$ such that every metric supported on $\mathcal F$ admits
a stochastic $D$-embedding into the family of metrics supported on $\mathcal G$.  We will
write $\mathcal F \stackrel{D}{\rightsquigarrow} \mathcal G$ if we wish to emphasize the
particular constant.  Finally, we write $\mathcal F \not\rightsquigarrow \mathcal G$
if no such $D$ exists.  The relationship with Conjecture \ref{conj:gnrs} is given
by the following simple lemma (see, e.g. \cite{GNRS99}).

\begin{lemma}\label{lem:randomreduce}
If $\mathcal F \stackrel{D}{\rightsquigarrow} \mathcal G$, then
$c_1(\mathcal F) \leq D \cdot c_1(\mathcal G)$.
\end{lemma}

At first glance, $\rightsquigarrow$ seems like a powerful operation; indeed, in \cite{GNRS99} it is proved
that $\mathsf{OuterPlanar} \embeds \mathsf{Trees}$, where $\mathsf{OuterPlanar}$ and $\mathsf{Trees}$
are the families of outerplanar graphs and connected, acylic graphs, respectively.
In general, if $L$ is a finite list of graphs, we will write $\mathcal E L$ for the
family of all graphs which do not have a member of $L$ as a minor.  The preceding
result can be restated as $\mathcal E\{K_{2,3}\} \embeds \mathcal E \{K_3\}$, where
$K_{n}$ and $K_{m,n}$ denote the complete and complete bipartite graphs, respectively.
Unfortunately, \cite{GNRS99} also showed that this cannot be pushed much further:
$\mathcal E \{K_4\} \not\rightsquigarrow \mathcal E\{K_3\}$.  Restated,
this means that even graphs of treewidth 2 cannot be stochastically embedded into trees.

These lower bounds were extended in \cite{CG04} to show that
$\mathsf{Treewidth}(k+3) \not\rightsquigarrow \mathsf{Treewidth}(k)$ for any $k \geq 1$.
Finally, in \cite{CJLV08}, these results are extended
to any family with a weak closure property, which we describe next.

\medskip
\noindent
{\bf Sums of graphs.}
We now introduce a graph operation which will be useful in stating our results.
Suppose that $H$ and $G$ are two graphs and $C_H, C_G$ are $k$-cliques in $H$ and $G$
respectively, for some $k \geq 1$.  One defines the {\em $k$-sum of $H$ and $G$} as the graph $H {\oplus_k} G$
which results from taking the disjoint union of $H$ and $G$ and then
identifying the two cliques $C_H$ and $C_G$, and possibly removing a subset
of the clique edges.  We remark that the notation is somewhat
ambiguous, as both the cliques and their identifications are implicit.
For a family
of graphs $\mathcal F$, we write $\oplus_k \mathcal F$ for the closure of $\mathcal F$
under $i$-sums for every $i=1,2,\ldots,k$.  With this notation in hand, we can state the following theorem.

\begin{theorem}[\cite{CJLV08}]
\label{thm:2sum}
If $\mathcal F$ and $\mathcal G$ are families of graphs and $\mathcal G$ is minor-closed,
then $\oplus_2 \mathcal F \rightsquigarrow \mathcal G$ implies $\mathcal F \subseteq \mathcal G$.
\end{theorem}

In fact, one case of this theorem relies on Theorem \ref{thm:nonembed} proved in the present paper, which states that for every $k = 1,2,\ldots,$ we have $\mathsf{Trees} \cap \mathsf{Pathwidth}(k+1) \not\rightsquigarrow \mathsf{Pathwidth}(k)$,
where $\mathsf{Pathwidth}(k)$ denotes the class of pathwidth-$k$ graphs (see Section \ref{sec:prelims} for the relevant
definitions).

Theorem \ref{thm:2sum} implies, for example, that $\mathsf{Planar} \cap \mathsf{Treewidth}(k+1) \not\rightsquigarrow \mathsf{Treewidth}(k)$
for any $k \geq 1$,
where $\mathsf{Planar}$ is the family of planar graphs,
since planar graphs and bounded treewidth
graphs are both closed under 2-sums.  The assumptions of the preceding theorem imply
that even random embeddings are not particularly useful for reducing the topology when $\oplus_2 \mathcal F = \mathcal F$.
However, some recent reductions suggest that when $\oplus_2 \mathcal F \neq \mathcal F$,
the situation is more hopeful.

In \cite{CGNRS06}, it is proved
that $\mathsf{Outerplanar}(k) \rightsquigarrow \mathsf{Trees}$.  Perhaps more surprisingly, it is shown
in \cite{IS07} that $\mathsf{Genus}(g) \rightsquigarrow \mathsf{Planar}$,
where $\mathsf{Genus}(g)$ is the family of graphs embedded on an orientable surface of genus $g$, and $\mathsf{Genus}(0) = \mathsf{Planar}$.
Note that while trees and planar graphs are closed under 2-sums, neither $\mathsf{Outerplanar}(k)$
nor $\mathsf{Genus}(g)$ are for $k \geq 1$ and $g \geq 1$.

It should be noted that an extensive amount of work has been done on embedding
finite metric spaces into distributions over trees, where the distortion
is allowed to depend on $n$, the number of points in the metric space;
see, e.g. \cite{Bartal96,Bartal98,FRT04}.  These results are not particularly
useful for us since we desire bounds that are independent of $n$.

\subsection{Results and techniques}

We now discuss the main results of the paper, along with the techniques
that go into proving them.

In \cite{GNRS99}, it is proved that $c_1(\mathsf{Treewidth}(2)) < \infty$, and later
works \cite{LR10,CJLV08} nailed down the precise dependence
$c_1(\mathsf{Treewidth}(2)) = 2$.  Resolving whether
$c_1(\mathsf{Treewidth}(3))$ is finite
seems quite difficult, and is a well-known open problem.  In fact, perhaps the simplest
``width $3$'' problem (which was open until the present work)
involves the family $\mathsf{Pathwidth}(3)$
(recall that $\mathsf{Pathwidth}(k) \subseteq \mathsf{Treewidth}(k)$
denotes the family of graphs of pathwidth at most $k$; see Section \ref{sec:prelims}).
These families are fundamental in the graph minor theory (see e.g.~\cite{RS83,Lov06});
see Lemma \ref{lem:compk} for an inductive definition.

Our main technical theorem shows that graphs of bounded pathwidth
can be randomly embedded into trees.  In fact, the theorem shows something
slightly stronger, that the target trees themselves can be taken
to have bounded pathwidth.

\begin{theorem}\label{thm:pwmain}
For every $k \in \mathbb N$, $\mathsf{Pathwidth}(k) \rightsquigarrow \mathsf{Trees} \cap \mathsf{Pathwidth}(k)$.
Quantitatively, $$\mathsf{Pathwidth}(k) \overset{D}{\rightsquigarrow} \mathsf{Trees} \cap \mathsf{Pathwidth}(k)\,,$$ for some $D \leq (4k)^{k^3+1}$.
\end{theorem}

In particular, this verifies Conjecture \ref{conj:gnrs} for graphs of bounded pathwidth.
The quantitative bound of Theorem \ref{thm:pwmain} is likely far from tight.
Naively, one might hope that for $D \leq O(\log k)$, one
has $\mathsf{Pathwidth}(k) \overset{D}{\rightsquigarrow} \mathsf{Trees}$.
But, in fact, known results imply that the distortion must
satisfy $D \geq \Omega(k)$.
The $k$-th level diamond graph (see \cite{GNRS99}) 
has pathwidth $O(k)$ but it is shown in \cite{GNRS99}
that every stochastic embedding of this graph
into a distribution over trees incurs distortion $\Omega(k)$.

Robertson and Seymour \cite{RS83} showed that a minor-closed family $\mathcal F$
excludes a forest if and only if $\mathcal F \subseteq \mathsf{Pathwidth}(k)$
for some $k \in \mathbb N$.

\begin{corollary}
If $T$ is any tree, then $\mathcal E\{T\} \rightsquigarrow \mathsf{Trees}$.
\end{corollary}

As a consequence of this, together with Lemma \ref{lem:randomreduce}, and the
elementary fact that $c_1(\mathsf{Trees})=1$, we resolve Conjecture \ref{conj:gnrs} whenever $\mathcal F$ forbids
some tree, yielding Theorem \ref{thm:pw}.
We remark that Theorem \ref{thm:pw} was unknown even for $\mathcal F = \mathsf{Pathwidth}(3)$.

\iffalse
\begin{theorem}
If $\mathcal F$ is any minor-closed family which doesn't contain all trees,
then $c_1(\mathcal F) < \infty$.  Equivalently, the multi-commodity max-flow/min-cut
gap for $\mathcal F$ is uniformly bounded, i.e. $\mathsf{gap}(\mathcal F) < \infty$.
\end{theorem}
\fi

In Section \ref{sec:nopath}, we complement our upper bound
by proving the following theorem.

\begin{theorem}\label{thm:nonembed}
For every $k \in \mathbb N$, $\mathsf{Pathwidth}(k+1) \cap \mathsf{Trees} \not\embeds \mathsf{Pathwidth}(k)$.
\end{theorem}

This result serves two purposes.  First, it shows that our proof of Theorem \ref{thm:pwmain},
which embeds $\mathsf{Pathwidth}(k)$ directly into trees cannot proceed by inductively
reducing the pathwidth by one.  Secondly, it is needed in the proof of Theorem \ref{thm:2sum}
in the case when $\mathcal F$ contains only trees (the techniques of \cite{CJLV08} handle
the case when $\mathcal F$ contains at least one cycle).
We remark that, perhaps surprisingly, the proof our non-embeddability result (Theorem \ref{thm:nonembed}) uses our embedding result (Theorem \ref{thm:pwmain}).

\subsection{Preliminaries}
\label{sec:prelims}

We now review some basic definitions and notions which
appear throughout the paper.

\medskip
\noindent
{\bf Graphs and metrics.}
We deal exclusively with
finite graphs $G = (V,E)$ which are free of loops
and parallel edges.
We will also write $V(G)$ and $E(G)$ for
the vertex and edge sets of $G$, respectively.
A {\em metric graph} is a graph $G$
equipped with a non-negative length function on edges $\len : E \to \mathbb R_+$.
We will denote the pseudometric
space associated with a graph $G$ as $(V,d_G)$, where $d_G$ is the
shortest path metric according to the edge lengths.
Note that $d_G(x,y)=0$ may occur even when $x \neq y$, and also
if $G$ is disconnected, there will be pairs $x,y \in V$ with
$d_G(x,y)=\infty$.  We allow both possibilities throughout the paper.
An important point is that {\em all length functions in the paper
are assumed to be reduced,} i.e. they satisfy
the property that for every $e = (u,v) \in E$, $\len(e) = d_G(u,v)$.

Given a metric graph $G$, we extend the length function to
paths $P \subseteq E$ by setting $\len(P) = \sum_{e \in P} \len(e)$.
For a pair of vertices $a,b \in P$, we use the notation $P[a,b]$
to denote the sub-path of $P$ from $a$ to $b$.
We recall that for a subset $S \subseteq V$, $G[S]$
represents the induced graph on $S$.  For a pair of subsets $S,T \subseteq V$,
we use the notations $E(S,T) = \{ (u,v) \in E : u \in S, v \in T \}$
and $E(S) = E(S,S)$.  For a vertex $u \in V$,
we write $N(u) = \{ v \in V :  (u,v) \in E \}$.

\medskip
\noindent
{\bf Graph minors.}
If $H$ and $G$ are two graphs, one says that
$H$ is a {\em minor} of $G$ if $H$ can be obtained from $G$
by a sequence of zero or more of the three operations:
edge deletion, vertex deletion, and edge contraction.
$G$ is said to be {\em $H$-minor-free} if $H$
is not a minor of $G$.  We refer to \cite{Lov06,DiestelBook}
for a more extensive discussion of the vast graph minor theory.

Equivalently, $H$ is a minor of $G$
if there exists a collection of disjoint sets $\{ A_v \}_{v \in V(H)}$
with $A_v \subseteq V(G)$ for each $v \in V(H)$,
such that each $A_v$ is connected in $G$, and there is an
edge between $A_u$ and $A_v$ whenever $(u,v) \in E(H)$.
A metric space $(X,d)$ is said to be {\em $H$-minor-free}
if it is supported on some $H$-minor-free graph.

\medskip
\noindent
{\bf Treewidth.} The notion of {\em treewidth} involves
a representation of a graph as a tree, called a tree
decomposition.  More precisely, a \emph{tree decomposition} of a
graph $G=(V, E)$ is a pair $(T, \chi)$ in which $T=(I, F)$ is a
tree and $\chi=\{\chi_i \mid i\in I\}$ is a family of subsets of
$V(G)$ such that (1) $\bigcup_{i\in I}\chi_i= V$; (2) for each
edge $e=\{u, v\} \in E$,
  there exists an $i\in I$ such that both $u$ and $v$
  belong to $\chi_i$; and (3) for all $v\in V$, the set of nodes
  $\{i\in I \mid v \in \chi_i\}$ forms a connected subtree of $T$.
To distinguish between vertices of the original graph $G$ and
vertices of $T$ in the tree decomposition, we call vertices of $T$
{\em nodes} and their corresponding $\chi_i$'s {\em bags}. The
maximum size of a bag in $\chi$ minus one is called the {\em width}
of the tree decomposition. The {\em treewidth} of a graph $G$
is the minimum width over all possible tree
decompositions of $G$.

\medskip
\noindent
{\bf Pathwidth.}
A tree decomposition is called a {\em path
decomposition} if  $T=(I, F)$ is a path.  The {\em pathwidth} of a
graph $G$ is the minimum width over all possible path decompositions
of $G$.  We will use the following alternate characterization.

\begin{definition}[Linear composition sequence]\label{defn:compk}
Let $k$ be a positive integer.
A sequence of pairs $(G_0, V_0), (G_1, V_1), \ldots, (G_t, V_t)$
is a {\em linear width-$k$ composition sequence for $G$} if $G_t = G$, $G_0$ is a $k$-clique with vertex set $V_0$, and
$(G_{i+1}, V_{i+1})$ arises from $(G_i, V_i)$ as follows:
Attach a new vertex $v_{i+1}$ to all the vertices of $V_i$ and choose $V_{i+1} \subseteq V_i \cup \{v_{i+1}\}$
so that $|V_{i+1}| = k$.
Observe that it is possible to have $V_{i+1}=V_i$.
We further note that for any $j\in \{1,\ldots,t\}$, we have $V(G_j) = V_0 \cup \{v_1,\ldots,v_j\}$.
\end{definition}

The following lemma is straightforward to prove.

\begin{lemma}\label{lem:compk}
A graph has pathwidth-$k$ if and only if it is a subgraph
of some graph possessing a linear width-$k$ composition sequence.
\end{lemma}
\begin{proof}[Proof sketch]
A path decomposition of width $k$ can be obtained from a width-$k$ composition sequence $(G_0,V_0),\ldots,(G_t,V_t)$ by setting for every $i\in \{1,\ldots,t\}$, the $i$-th bag to be $V_{i-1}\cup \{v_i\}$.
For the other direction, one can always assume that a pathwidth-$k$ graph admits a path depomposition of width $k$ such that every bag has size exactly $k+1$, and every two bags differ in exactly one vertex.
This immediately yields a linear width-$k$ composition sequence.
\end{proof}

\medskip
\noindent
{\bf Asymptotic notation.}
For two expressions $E$ and $F$, we sometimes use the notation $E \lesssim F$
to denote $E = O(F)$.  We use $E \approx F$ to denote the
conjunction of $E \lesssim F$ and $E \gtrsim F$.

%\item $\lesssim, \gtrsim$
%\item $N(a)$
%\item $\pi[a,b]$ notation for paths

\section{Warm-up:  Embedding pathwidth-2 graphs into trees}
\label{sec:warmup}

In this section, we prove that $\mathsf{Pathwidth}(2) \rightsquigarrow \mathsf{Trees}$,
as a warmup for the general case in Section \ref{sec:pathk}.
The pathwidth-$2$ case does not possess many of the difficulties
of the general case; in particular, it does not
require us to bound the stretch in multiple phases
(for which we introduce a rank parameter in the next section).
But it does show the importance of using
an inflation factor to blowup small edges,
in order for a certain geometric sum to converge.

\medskip

Let $G = (V,E)$ be a metric graph of pathwidth 2.
By Lemma \ref{lem:compk}, it suffices to give a probabilistic embedding for a graph $G$ possessing a linear width-$2$ composition sequence $(G_0, e_0), \ldots, (G_t, e_t)$, where $e_i$ plays the role of $V_i$ in Definition \ref{defn:compk}.
We will inductively embed $G$ into a distribution over its spanning trees.
First, we put $T_0=e_0$.
Now, let $T_i$ be a spanning tree of $G_i$, with $e_i \in E(T_i)$.
We will produce a random spanning tree $T_{i+1}$ of $G_{i+1}$ with $e_{i+1} \in E(T_{i+1})$ as follows.
Let $e_i = \{u,v\}$, and let $w^*$ be the newly attached vertex.
We also add the edges $\{u,w^*\}$, and $\{v,w^*\}$, so the resulting graph is not a tree.
We obtain a tree by randomly deleting either $\{u,w^*\}$, or $\{v,w^*\}$ as follows.
Let $\tau =12$; we refer to this constant as an ``inflation factor.''

There are two cases.

\begin{enumerate}
\item
If $e_i = e_{i+1}$, we
delete $\{u,w^*\}$ with probability $\frac{\len(u,w^*)}{\len(u,w^*)+\len(v,w^*)}$,
and otherwise we delete $\{w^*, v\}$.
\item
If $e_i \neq e_{i+1}$, assume (without loss of generality) that $e_{i+1} = \{v,w^*\}$.
In that case, we delete $\{u,w^*\}$ with probability
\begin{equation}\label{eq:prob}
\min\left\{\frac{\tau \len(u,w^*)}{\len(u,w^*) + \len(u,v)},1\right\},
\end{equation}
and otherwise we delete $\{u,v\}$.
\end{enumerate}
It is easy to see that if $T_i$ was a spanning tree, then so is $T_{i+1}$.
Furthermore, by construction $e_{i+1} \in E(T_{i+1})$.
%It suffices now to bound the expected stretch in $T_i$.
Let $T = T_t$ be the final tree, and set $T_i=T$ for $i > t$.
It remains to bound the expected stretch in $T$.

\remove{
To this end, we will define an additional pseudometric $d_i$ on $V$ as follows:
$$
d_i(x,y) = \begin{cases}
d_{T_i}(x,y) & x,y \in V(G_i) \\
d_G(x,y) & x,y \notin V(G_i) \\
d_{T_i}(x, z) + d_G(z, y) & x \in V(G_i), y \notin V(G_i)\,,
\end{cases}
$$
where $z \in \{u,v\}$ is such that $x$ is in the same component as $z$ in $T_i \setminus e_i$.
In other words, $d_i(x,y) = d_{G \setminus E_i}(x,y)$, where $E_i = E(G_i) \setminus E(T_i)$
is the set of edges we have removed through stage $i$.
This follows because the edge $\{u,v\}$ is present in $T_i$,
our length function on $G$ is reduced, and $\{u,v\}$ separates $V(G_i) \setminus \{u,v\}$ from $V \setminus V(G_i)$.

\medskip
}

For every edge $\{x,y\} \in E(G_i)$ and $i \geq 0$,
define the value,
$$
K_i^{x,y} = \max \left\{ \E\left[\frac{d_T(x,y)}{d_{T_i}(x,y)} \,\Big|\, T_i = \Gamma \right] : \pr(T_i=\Gamma) > 0 \right\}\,.
$$
This is the maximum expected stretch between $x$ and $y$ incurred over all stages later than $i$,
conditioned on the worst possible configuration for $T_i$.

For each $x \in V$, define $s(x)=-1$ for $x \in V(G_0)$, and otherwise
it is the unique value $s \geq 0$ such that $x \in V(G_{s+1}) \setminus V(G_s)$.
Also define $s(x,y)=\max(s(x),s(y))$.
The next two lemmas form the core of our analysis.

\begin{lemma}\label{lem:pwd2_first}
If $\{x,y\} \in E$ and $s(x,y)=i$, then
\begin{equation}\label{eq:pwd2}
\E\left[d_T(x,y) \right] \leq 3\tau \cdot K_{i+1}^{x,y} \cdot \len(x,y)\,.
\end{equation}
\end{lemma}

\begin{proof}
If $x,y \in V(G_0)$, then $s(x,y)=-1$ and $\E [d_T(x,y)]=K_0^{x,y} \cdot \len(x,y)$ by definition.
Otherwise, assume without loss of generality that $s(x) < s(y)$.
In this case, it must be that $x \in e_i = \{u,v\}$ and $y = w^*$.
Suppose that $x = u$.

If $e_{i+1}=e_i$, an elementary calculation based on case (1) of our algorithm yields,
$$
\E \left[ \frac{d_{T_{i+1}}(u,w^*)}{\len(u,w^*)} \,\Big|\, T_i\right] \leq \frac{3\, \len(v,w^*) + \len(u,w^*)}{\len(v,w^*)+\len(u,w^*)} \leq 3\,,
$$
from which $\E [d_T(u,w^*)] \leq 3\, K_{i+1}^{u,w^*} \cdot \len(u,w^*)$ immediately follows.

Similarly, if $e_{i+1}=\{v,w^*\}$, then the expected stretch is inflated by at most a factor of $\tau$, and therefore \eqref{eq:pwd2} again follows
by a similar calculation.
Finally, if $e_{i+1}=\{u,w^*\}$, then $\{u,w^*\}\in E(T_{i+1})$, and therefore $\E[d_T(u,w^*)] \leq K_{i+1}^{u,w^*} \cdot \len(u,w^*)$.
\end{proof}

\begin{lemma}\label{lem:pwd2_next}
For any $\{x,y\} \in E(G_i)$, we have $K_i^{x,y}\leq \max\{3,K_{i+1}^{a,b}\}$ for some $\{a,b\} \in E(G_i)$.
\end{lemma}
\begin{proof}
Let $\Gamma$ be a tree on $V(G_i)$ which is a maximizer for $K_i^{x,y}$.
Let $\Gamma_u$ and $\Gamma_v$ be the subtrees of $\Gamma \setminus e_i$ rooted at $u$ and $v$ respectively, where
we recall that $e_i = \{u,v\}$.
If $x$ and $y$ are both either in $\Gamma_u$, or in $\Gamma_v$, then $K_i^{x,y}=1$,
since $\Gamma_u$ and $\Gamma_v$ remain intact in the final tree $T$, conditioned on $T_i=\Gamma$.

So, it suffices to consider the case $x \in \Gamma_u$ and $y \in \Gamma_v$.
Observe further that since the unique path between $x$ and $y$ in $\Gamma$ passes through $\{u,v\}$, and the $x$-$u$ and $y$-$v$ paths
will both remain in $T$, we have
$$
\E\left[\frac{d_T(x,y)}{d_{T_i}(x,y)} \,\Big|\, T_i = \Gamma \right]
\leq
\E\left[\frac{d_T(u,v)}{d_{T_i}(u,v)} \,\Big|\, T_i = \Gamma \right] \leq
K_i^{u,v}\,.
$$

Thus to prove the lemma, it suffices to show that $K_i^{u,v} \leq \max \{3, K_{i+1}^{u,v}\}$.
To this end, let $\Gamma$ be the maximizer for $K_i^{u,v}$, and suppose that $T_i=\Gamma$.
%To this end, now assume that $T_i=\Gamma$, where $\Gamma$ satisfies $K_i^{u,v}(\Gamma)=K_i^{u,v}$.
If $e_{i+1} = e_i$, then the edge $\{u,v\}$ remains intact (i.e.~$\{u,v\}\in E(T_{i+1})$), and therefore $K_i^{u,v} \leq K_{i+1}^{u,v}$.
Assume now that $e_{i+1}\neq e_i$, which means that we are in case (2) of the algorithm.
Assume further, without loss of generality, that $e_{i+1}=\{v,w^*\}$.
Recall that either $\{u,v\}$ or $\{u,w^*\}$ is deleted.
%If $\{u,w^*\}$ is deleted, then $d_T(u,v)=d_{T_i}(u,v)$.
%If, on the other hand, $\{u,v\}$ is deleted, then the new path between $u$ and $v$ is $u-w^* - v$.

%To make sure that the probabilities work out we inflate the probability of cutting $(u,w^*)$ (via the inflation factor).  This incurs a bigger one-time cost when an edge is first cut, but saves us in the recursive stretch.

%So now consider $x \in T_u$ and $y \in T_v$.  (Note, we may assume that \eqref{eq:prob} is less than one, otherwise we are always in the good case where $(u,w^*)$ is cut.  In this case, $(u,w^*)$ only incurs (deterministic) stretch at most $2\tau$.)

%To see the stretch that $x,y$ incur from the future, it suffices to look at the stretch of $u$ and $v$.
Let $A = \len(u,w^*), B = \len(u,v), C = \len(v,w^*)$.
%(it helps to draw a picture).
With probability $p=\min\{1,\frac{\tau A}{A+B}\}$, the edge $\{u,w^*\}$ is deleted, in which case $d_T(u,v)=d_{T_i}(u,v)$.
With probability $1-p$, the edge $\{u,v\}$ is deleted, and the new path between $u$ and $v$ in $T_{i+1}$ is $u$-$w^*$-$v$,
so the distance between $u$ and $v$ is stretched to $A+C \leq 2A+B$,
and is eligible to be stretched by at most a factor $K_{i+1}^{u,v}$ in the future.

Thus, if $A \geq B/(\tau-1)$, we have $K_{i}^{u,v}=1$.
We can therefore assume $A < B/(\tau-1)$.
%So let's assume that $A = \varepsilon B$ and $\varepsilon \leq 1/\tau$.
Thus we can bound,
\begin{eqnarray*}
K_i^{u,v} &\leq& \frac{\tau A}{A+B} + K_{i+1}^{u,v} \left(1-\frac{\tau A}{A+B}\right)\frac{2A+B}{B} \\
&\leq& \tau\frac{A}{B} + K_{i+1}^{u,v} \left(1-\frac{\tau A}{2 B}\right) \left(1+2\frac{A}{B}\right) \\
%&=& \tau\frac{A}{B} + K_{i+1} \left(1-\frac{\varepsilon \tau}{2}+2\varepsilon-\varepsilon^2\tau\right) \\
&\leq & \tau\frac{A}{B} + K_{i+1}^{u,v} \left(1-\frac{\tau A}{3 B}\right),
\end{eqnarray*}
where we have used $1-\frac{\tau A}{2B}+2\frac{A}{B} \leq 1-\frac{\tau A}{3B}$ since $\tau=12$.
%where in the final inequality we have assumed that $\tau$ is big enough.
%(and hence $\varepsilon$ is small enough).
But now one sees that,
%$$
%K_i^{u,v} \leq \tau \frac{A}{B}(1-K_{i+1}^{u,v}/3) + K_{i+1}^{u,v} \leq \max\{3,K_{i+1}^{u,v}\}.
%$$
$$
K_i^{u,v} \leq 3 \frac{\tau A}{3B} + K_{i+1}^{u,v} \left(1-\frac{\tau A}{3B} \right) \leq \max\{3,K_{i+1}^{u,v}\}.
$$
\end{proof}

Finally, the next lemma completes our analysis.

\begin{lemma}
For any $x,y\in V$, we have $\E[d_T(x,y)]\leq 9\tau \cdot d_G(x,y)$.
\end{lemma}

\begin{proof}
By the triangle inequality and linearity of expectation, it suffices to prove the lemma for edges $\{x,y\} \in E$.
We will prove the following by reverse induction on $i$:
For every $\{x,y\} \in E$ and $i \geq s(x,y)+1$, we have $K_{i}^{x,y} \leq 3$.
Combining this with Lemma \ref{lem:pwd2_first} will complete the proof.

The claim is trivial for $i=t$ since $K_t^{x,y}=1$ for all $\{x,y\} \in E$.
If $t > i \geq s(x,y)+1$, then $\{x,y\} \in E(G_i)$, and
Lemma \ref{lem:pwd2_next} immediately implies that $K_i^{x,y} \leq \max \{3, K_{i+1}^{a,b}\}$
for some $a,b$ with $i+1 \geq s(a,b)+1$.
By induction, $K_{i+1}^{a,b} \leq 3$, hence $K_i^{x,y} \leq 3$ as well.
\end{proof}

\section{Embedding pathwidth-$k$ graphs into trees}
\label{sec:pathk}

We now turn to graphs of pathwidth $k$ for some $k \in \mathbb N$.
Let $G$ be such a graph.
By Lemma \ref{lem:compk}, we may assume that
$G$ has a linear width-$k$ composition sequence, $(G_0,V_0),\ldots(G_t,V_t)$.
For $i \geq 1$, we define $\widehat V_i = V_{i-1} \cup \{v_i\}$.
Our algorithm for embedding $G$ into a random tree proceeds inductively along the composition sequence.
For each $i\in \{1,\ldots,t\}$ we compute a subgraph $H_i$ of $G_i$, whose only non-trivial 2-connected component is a $(k+1)$-clique on $\widehat V_i$ (see Figure \ref{fig:Hi}).
More specifically, $H_1$ is just a clique on $\hat V_1$.
Given $H_i$, we derive $H_{i+1}$ by adding all the edges between $v_{i+1}$ and $V_i$, and removing all the edges, except for one, between
$V_i$ and the unique vertex in $\widehat V_{i} \setminus V_{i}$.

\begin{figure}
\begin{center}
\includegraphics{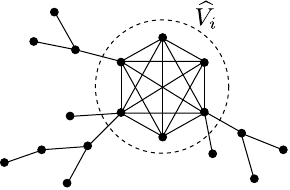}
\caption{The graph $H_i$ with $k=5$.\label{fig:Hi}}
\end{center}
\end{figure}

The main part of the algorithm involves determining which edge in $(\widehat V_{i} \setminus V_{i}) \times V_{i}$ we keep in $H_{i+1}$.
The high-level idea behind our approach is as follows.
On one hand, we want to keep short edges so that the distance between $\widehat V_{i} \setminus V_{i}$ and $V_{i}$ is small.
On the other hand, keeping always the shortest edge leads to accumulation of the stretch
for certain pairs (whose shortest-path keeps getting longer, through a sequence of ``short'' edges).
We avoid this obstacle via a randomized process that assigns a \emph{rank} to each edge, which intuitively means that edges of lower rank are more likely to be deleted.
More specifically, at each step $i$, we pick a random threshold $L$ and keep the highest ranked edge of length at most $L$, deleting the rest.
We also update the ranks of the edges in the new graph appropriately.

Formally, let $\rank_i : V(G)\times V(G)\to \mathbb Z_{\geq 0}$ be an arbitrary symmetric function, with $\rank_1(u,v)=0$, for each $u,v\in V(G)$.
Let $E(\widehat V_i)={\widehat V_i \choose 2}$, i.e.~the set of edges internal to $\widehat V_i$.
For $u,v\in V(H_i)$, let $P_i^{u,v}$ be the unique path between $u$ and $v$ in $H_i$ that contains at most one edge in $E(\widehat V_i)$.
Observe that $P_i^{u,v}$ is well-defined since $\widehat V_i$ forms a clique.
For an edge $e\in E(\widehat V_i)$ we set
\[
\edgerank_i(e) = \max_{u,v\in V(H_i) : e\in P_i^{u,v}} \rank_i(u,v)
\]

The randomized process for generating $H_{i+1}$ and $\rank_{i+1}$ from $H_i$ and $\rank_i$ is as follows.
%We describe how to randomly generate a graph $H_{i+1}$ and rank function $\rank_{i+1} : E(V_{i+1}) \to \mathbb Z_{\geq 0}$.
%Let $\tau \geq 1$ be a {\em inflation factor} which we set later.
Let $\tau = 4k$ be our new ``inflation factor.''

%\medskip
\begin{quote}
Let $w$ be the unique vertex in $\widehat V_i \setminus V_{i}$,
and enumerate $E(w, V_{i}) = \{e_1, e_2, \ldots, e_k\}$
so that $\len(e_1) \leq \len(e_2) \leq \cdots \leq \len(e_k)$.

Now, let $\{\sigma_j\}_{j=1}^{k-1}$ be a family of independent
$\{0,1\}$ random variables with
$$
\pr[\sigma_j = 1] = \min\left\{1,\tau \frac{\len(e_j)}{\len(e_{j+1})}\right\}\,,
$$
and define the set of eligible edges by
$$
\mathcal E = \left\{ e_j : \prod_{i=1}^{j-1} \sigma_i = 1 \right\}\,.
$$
In particular, $e_1 \in \mathcal E$ always.
Let $e^* \in \mathcal E$ be any edge satisfying $\edgerank_i(e^*) = \max_{e \in \mathcal E} \edgerank_i(e)$.

Finally, we define $H_{i+1}$ as the graph with vertex set $V(G_{i+1})$
and edge set (see Figure \ref{fig:transition}),
$$E(H_{i+1})=\{e^*\} \cup \{ \{v_{i+1}, u \} : u \in V_{i} \} \cup \left(E(H_i) \setminus E(w, V_{i})\right).$$
%In other words, we remove everything adjacent to $w$ except $e^*$, and we form a clique on $V_{i+1}$.

We also define $\rank_{i+1}$ as follows.
For any $u,v\in V(G)$
$$
\rank_{i+1}(u,v) = \begin{cases}
\rank_i(u,v) & \textrm{if } E(P_i^{u,v})\cap {\cal E}=\emptyset\\
\rank_i(u,v)+1 & \textrm{otherwise }.
\end{cases}
$$
\end{quote}

Intuitively, $\rank_i(u,v)$ counts how many times the path between $u$ and $v$ was under risk to be {\em significantly} stretched until step $i$.
If $P_i^{u,v}$ does not use an edge of $\mathcal E$ then the edge $\{u,v\}$ will be stretched,
but the alternative path will, on average, be ``short enough'' that
we need not increase its rank (this is how the set $\mathcal E$ is defined).
It remains to analyze the expected stretch incurred by the above process.
First, we observe that the maximum rank of an edge is $O(k^2)$.

\begin{figure}
\begin{center}
\includegraphics{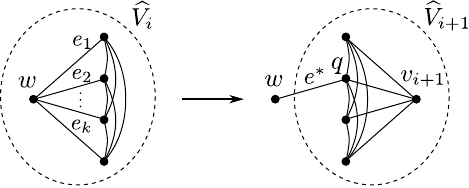}
\caption{Transitioning from $H_i$ to $H_{i+1}$. Here, $w$ is the unique vertex in $\widehat{V}_i \setminus V_{i}$. \label{fig:transition}}
\end{center}
\end{figure}

\begin{lemma}\label{lem:edgerank}
For every $i = 1, 2, \ldots, t$ and every edge $e \in E(\widehat V_{i})$, $\edgerank_i(e) \leq {k+1 \choose 2}$.
\end{lemma}
\begin{proof}
For each $i=1,2,\ldots,t$, and each $j=1,2,\ldots,{k+1 \choose 2}$, let $R_{i,j}$ be the $j$-th largest edge-rank of the edges in $E(\widehat{V}_{i})$.
That is, for each $i=1,2,\ldots,t$,
$R_{i,1}\leq R_{i,2}\leq \ldots \leq R_{i,{k+1 \choose 2}}$.

We will prove by induction on $i$ that for each $i=1,2,\ldots, t$,
for each $1 \leq j \leq {k+1 \choose 2}$, we have $R_{i,j} \leq j$.
For $i=1$, all the ranks are equal to $0$, and the assertion holds trivially.

Assume now that the assertion holds for $i-1$.
It is convenient to analyze the transition from step $i-1$ to step $i$ in three phases.
We need to remove the edges in $E(w, V_i) \setminus \{e^*\}$  and add the edges in $E(v_{i+1}, V_i)$, while updating the ranks accordingly.
For notational simplicity, we assume that the rank of an edge that is removed is set to zero.
Let $e^*$ be the maximum-rank edge in ${\cal E}$.
In the first phase, we set the rank of $e^*$ to zero, and we increase the rank of all remaining edges in ${\cal E}$ by one.
Clearly, the resulting edge ranks satisfy the inductive invariant.

In the second phase, for any edge $e\in E(w, V_i) \setminus \{e^*\}$, we update the rank of an edge $e'=e'(e)\in E(\widehat V_{i})\cap E(\widehat V_{i+1})$ to be $\edgerank(e')=\max\{\edgerank(e'),\edgerank(e)\}$, and we set the rank of $e$ to zero.
The point here is that for any $e \in E(w,V_i)$, there is a unique edge $e' \in E(\widehat V_{i})\cap E(\widehat V_{i+1})$
such that, for any $u,v \in V(H_i)$, if $e \in P_i^{u,v}$ then $e' \in P_{i+1}^{u,v}$.
In other words, the paths that use $e$ will have to be rerouted through a new path
that uses $e'$.  This explains how the edge-rank of $e$ is ``transferred'' to $e'$.

Clearly, after the second phase the ranks still satisfy the inductive invariant.
Finally, in the third phase we remove the edges in $E(w,V_i)$, and we add the edges in $E(v_{i+1}, V_i)$.
All the removed edges have at this point rank zero, and all new edges also have rank zero.
Thus, the inductive invariant is satisfied.
%Assume now that the assertion holds for $i=l-1$, and consider the case $i=l$.
%Let $j\in \{1,\ldots,{k \choose 2}\}$.
%If $R_{i,j}=0$, then there is nothing to show.
%Otherwise, assume that $R_{i,j}>0$.
%Observe that the maximum rank $R_{i-1,{k \choose 2}}$ is removed, so $R_{i,j}$ is obtained by adding at most 1 to some rank $R_{i-1, j'}$, for some $j'\leq j-1$.
%Therefore, $R_{i,j}\leq R_{i-1,j'}+1 \leq j'+1 \leq j$.
%We obtain that for each $i\in \{1,\ldots,t\}$ the maximum rank is $R_{i,{k\choose 2}} \leq {k\choose 2}$, and the lemma follows.
\end{proof}

\iffalse
Now, since the rank of any edge is bounded, it means that a pair of points can only experience
stretch in one of $O(k^2)$ phases (a phase is basically the time between the appearance of
the appropriate edge in $\mathcal E$).  The important thing is to show that {\em within} a phase,
the total stretch is $O(1)$ (possibly depending on $k$).  Given the setup of the embedding,
this part of the analysis is like the pathwidth $2$ case.  The idea is then when we are eligible,
we are willing to lose a large expansion factor because our rank will increase, and this can
only happen a bounded number ($O(k^2)$) of times.  When we are not eligible, we have to be careful
that the expansions do not build up until the next eligible time.  Here is the analysis.
\fi

\medskip

For any $i\in \{1,\ldots,t\}$, $r\in \{0,\ldots,{k+1 \choose 2}\}$,
and any edge $\{u,v\} \in E(G_i)$, we put
%if $H_t$ denotes the final computed random graph, then
\begin{equation}\label{eq:kir}
K_i^{u,v}(r) = \max \left\{ \mathbb{E}\left[\frac{d_{H_t}(u,v)}{d_{H_i}(u,v)} \, \Big|\, H_i=\Gamma, \rank_i(u,v)=\rho \right] : (\Gamma, \rho) \in \Omega_i(u,v;r) \right\},
\end{equation}
where we define
\[
\Omega_i(u,v;r) = \left\{ (\Gamma, \rho) : \pr(H_i=\Gamma, \rank_i(u,v)=\rho) > 0 \textrm{ and } \rho \geq r \right\}\,.
\]
In other words, $K_i^{u,v}(r)$ is the maximum expected stretch for all stages after $i$,
conditioned on the worst possible configuration over subgraphs $H_i$ and rank functions satisfying $\rank_i(u,v) \geq r$.
We further define $K_i^{u,v}\left({k+1 \choose 2}+1\right)=1$.

\medskip

For the next three lemmas and the corollary that follows,
we fix an edge $\{u,v\} \in E(G_i)$, and a number $r \in \{0,\ldots,{k+1 \choose 2}\}$.
Let $(\Gamma,\rho) \in \Omega_i(u,v;r)$ be a maximizer in \eqref{eq:kir}, and write
$\pr^*[\cdot] = \pr[\cdot\mid H_i=\Gamma, \rank_i=\rho]$
and
$\E^*[\cdot] = \E[\cdot\mid H_i=\Gamma, \rank_i=\rho]$.
A major point is that
the following calculations are oblivious to the conditioning, aside from
the assumption that $\rank_i(u,v) \geq r$.

%Moreover, let
%\[
%K_i(r) = \max_{u,v\in V(H_i)} K_i^{u,v}(r).
%\]

%\[
%K_i(r) = \max_{u,v\in V(H_i): \rank_i(u,v)=r} \frac{\mathbb{E}[d_{H_t}(u,v)]}{d_{H_i}(u,v)}.
%\]

\begin{lemma}\label{lem:pwd_master}
Suppose that $e_j \in E(P_i^{u,v})$ for some $j \in \{1,2,\ldots,k\}$.  Then,
\begin{eqnarray*}
K_i^{u,v}(r)
&\leq &
\pr^*\left[e_j \in \mathcal E\right] \ \left(1+2\frac{\mathbb E^*[\len(e^*)\,|\,e_j \in \mathcal E]}{\len(e_j)}\right) K^{u,v}_{i+1}(r+1)\\
 & & +\, \pr^*\left[e_j \notin \mathcal E\right] \left(1+2 \frac{\mathbb E^*[\len(e^*)\,|\,e_j \notin \mathcal E]}{\len(e_j)}\right) K^{u,v}_{i+1}(r) %\label{eq:master}
\end{eqnarray*}
\end{lemma}
\begin{proof}
%If $P_i^{u,v}$ does not contain an edge from $(V_i\setminus V_{i+1})\times V_i$, then $K_i(r)=K_{i+1}(r)$.
%Otherwise, let $e_j$ be the unique edge of $P_i^{u,v}$ in $(V_i\setminus V_{i+1})\times V_i$.
We have,
$\frac{d_{H_{i+1}}(u,v)}{d_{H_{i}}(u,v)} \leq \frac{2 \len(e^*) + \len(e_j)}{\len(e_j)}.$
There are two possibilities:  (1) $e_j \in \mathcal E$
occurs, and the rank of $\{u,v\}$ is increased by 1,
(2) $e_j \notin \mathcal E$ and the rank of $\{u,v\}$ remains the same.
This verifies the claimed inequality for $r < {k+1 \choose 2}$.

Note that, by Lemma \ref{lem:edgerank}, $\rank_i(u,v) \leq {k+1 \choose 2}.$
Thus
the lemma holds true even for $r={k+1 \choose 2}$, in which case
$e_j \in \mathcal E \implies e_j=e^*$ (since the rank of the pair $u,v$ cannot increase anymore).
If this happens, then
$d_{H_t}(u,v)=d_{H_i}(u,v)$, again verifying the claimed inequality,
since $K_{i+1}^{u,v}\left({k+1 \choose 2}+1\right)=1$ by definition.
\end{proof}

\begin{lemma}\label{lemma:inE}
For any $j\in [k]$,
$\pr^*\left[e_j \in \mathcal E\right] \left(1+2\frac{\mathbb E^*[\len(e^*)\,|\,e_j \in \mathcal E]}{\len(e_j)}\right) \leq 3 (4k)^{k-1} \frac{\len(e_1)}{\len(e_j)}$.
\end{lemma}
\begin{proof}
We have
\begin{eqnarray*}
\pr^*\left[e_j \in \mathcal E\right] \mathbb E^*[\len(e^*)\,|\,e_j \in \mathcal E] &\leq &
\mathbb E^*[\len(e^*)] \\
&=&
\sum_{h=1}^k \len(e_h) \pr^*[e^*=e_h] \\
&\leq &
\sum_{h=1}^k \len(e_h) \pr^*[e_h\in {\cal E}] \\
&\leq &
\sum_{h=1}^k \len(e_h) \frac{\len(e_1)}{\len(e_h)} \tau^{h-1} \\
&\leq &
2\, \len(e_1) \tau^{k-1}\,.
\end{eqnarray*}
Also, we have $\pr^*\left[e_j \in \mathcal E\right] \leq \tau^{j-1} \frac{\len(e_1)}{\len(e_j)} \leq \tau^{k-1} \frac{\len(e_1)}{\len(e_j)}$.
Combining these estimates yields the claim, recalling that $\tau=4k$.
\end{proof}

\begin{lemma}\label{lemma:notinE}
For any $j\in [k]$,
$\pr^*\left[e_j \notin \mathcal E\right] \left(1+2 \frac{\mathbb E^*[\len(e^*)\,|\,e_j \notin \mathcal E]}{\len(e_j)}\right) \leq 1-\frac{\len(e_1)}{\len(e_j)}.$
\end{lemma}
\begin{proof}
Let $$I = \left\{h\in \{1,2,\ldots,j-1\} : \len(e_{h+1}) > \tau \cdot \len(e_h)\right\}.$$
Observe that if $h\in \{1,2,\ldots,j-1\}\setminus I$, then whenever $e_h \in \mathcal E$, we have also $e_{h+1}\in \mathcal E$.
For each $h\in I$, let $k_h=|I\cap \{1,2,\ldots,h\}|$.
\begin{eqnarray*}
\pr^*\left[e_j \notin \mathcal E\right] \left(1+2 \frac{\mathbb E^*[\len(e^*)\,|\,e_j \notin \mathcal E]}{\len(e_j)}\right)
&\leq& \sum_{h=1}^{j-1} \pr^*\left[e_h \in \mathcal E \textrm{ and } e_{h+1} \notin \mathcal E\right] \left(1+\frac{2\len(e_h)}{\len(e_j)}\right) \\
&=& \sum_{h\in I} \pr^*\left[e_h \in \mathcal E \textrm{ and } e_{h+1} \notin \mathcal E\right] \left(1+\frac{2\len(e_h)}{\len(e_j)}\right) \\
&\leq &
\sum_{h\in I} \tau^{k_h-1} \frac{\len(e_1)}{\len(e_{h})} \left(1-\tau \frac{\len(e_{h})}{\len(e_{h+1})}\right) \left(1+\frac{2\len(e_h)}{\len(e_j)}\right) \\
&\leq & \sum_{h\in I} \tau^{k_h-1} \frac{\len(e_1)}{\len(e_h)} \left(1-\tau \frac{\len(e_{h})}{\len(e_{h+1})} + \frac{2\len(e_h)}{\len(e_j)}\right) \\
&=& \sum_{h\in I} \tau^{k_h-1} \frac{\len(e_1)}{\len(e_h)} \left(1-\tau \frac{\len(e_{h})}{\len(e_{h+1})}\right)
+  \frac{\len(e_1)}{\len(e_j)} \sum_{h\in I} 2 \tau^{k_h-1}  \\
%&\leq & 1- \tau^{|I|} \frac{\len(e_1)}{\len(e_{h_{|I|}+1})} + \frac{\len(e_1)}{\len(e_j)} (2k\tau^{|I|-1})  \\
&\leq & 1- \tau^{|I|} \frac{\len(e_1)}{\len(e_{j})} + \frac{\len(e_1)}{\len(e_j)} (2k\tau^{|I|-1})  \\
& = & 1+\frac{\len(e_1)}{\len(e_j)} \left(2k\tau^{|I|-1}-\tau^{|I|}\right)\\
& \leq & 1-\frac{\len(e_1)}{\len(e_j)}.
\end{eqnarray*}
\end{proof}

\begin{corollary}\label{cor:major}
For every $\{u,v\} \in E(G_i)$ and $r \in \{0,\ldots,{k+1 \choose 2}\}$, we have
$$
K_i^{u,v}(r)
\leq
\max\left\{ 3(4k)^{k-1} K^{u,v}_{i+1}(r+1), K^{u,v}_{i+1}(r) \right\}.
$$
\end{corollary}

\begin{proof}
Suppose that $H_i=\Gamma$ and $\rank_i=\rho$.  If $E(P_i^{u,v}) \cap E(\hat V_i)$ is empty,
then $K_i^{u,v}(r)=1$ because the current $u$-$v$ path in $H_i$ will be preserved in $H_t$.
Otherwise, we have $E(P_i^{u,v}) \cap E(\hat V_i) = \{e_j\}$ for some $j \in [k]$.
Apply Lemmas \ref{lem:pwd_master}, \ref{lemma:inE}, and \ref{lemma:notinE}
to conclude that
\begin{eqnarray*}
K^{u,v}_i(r)
&\leq &
 \frac{\len(e_1)}{\len(e_j)} 3(4k)^{k-1} K^{u,v}_{i+1}(r+1) +
\left(1-\frac{\len(e_1)}{\len(e_j)}\right) K^{u,v}_{i+1}(r) \\
%& = & \frac{\len(e_1)}{\len(e_j)} \left((4k)^k K^{u,v}_{i+1}(r+1) - K^{u,v}_{i+1}(r) \right) + K^{u,v}_{i+1}(r) \\
&\leq &
\max\left\{ 3(4k)^{k-1} K^{u,v}_{i+1}(r+1), K^{u,v}_{i+1}(r) \right\},
\end{eqnarray*}
completing the proof.
\end{proof}

We can now state and prove our main theorem.

\begin{theorem}\label{thm:pathwidth-into-trees}
For every $k \geq 1$,
every metric graph of pathwidth $k$ admits a stochastic $D$-embedding into a distribution
over trees with $D \leq (4k)^{k^3}$.
\end{theorem}

\begin{proof}
We may assume that $k \geq 2$ as the statement is trivial for $k=1$.
Let $H_t$ be the random subgraph of $G$.
Fix $\{u,v\} \in E(G)$, and suppose that $i_0$ is the smallest
number for which $u,v \in V(G_{i_0})$.
In this case, since $\{u,v\}$ is an edge, we have $d_{G_{i_0}}(u,v)=d_G(u,v)$, thus
$$
\E\left[d_{H_t}(u,v)\right] \leq K_{i_0}^{u,v}(0) \cdot \len(u,v)\,.
$$

Now applying Corollary \ref{cor:major} inductively immediately yields the bound,
$$
K_{i_0}^{u,v}(0) \leq \left(3(4k)^{k-1}\right)^{{k+1 \choose 2}+1},
$$
recalling that $K_i^{u,v}\left({k+1 \choose 2}+1\right)=1$ for all $i$, and
$K_t^{u,v}(r)=1$ for all $r$.

Finally, observe that the only non-trivial 2-connected component of $H_t$ is a $(k+1)$-clique on $\widehat V_t$.
Replacing $\widehat V_t$ by a minimum spanning tree yields a tree $T$ with $d_T(u,v) \leq (k+1)\cdot d_{H_t}(u,v)$.
This completes the proof.
\end{proof}

\section{$\mathsf{Pathwidth}(k+1) \not\embeds\mathsf{Pathwidth}(k)$}
\label{sec:nopath}

We now show that for any fixed $k\geq 1$, and for any $n\geq 1$, there exists an $n$-vertex graph of pathwidth $k+1$ for
which any stochastic $D$-embedding into graphs of pathwidth $k$ has $D \geq \Omega(n^{2^{-k}})$,
where the $\Omega(\cdot)$ notation hides a multiplicative constant depending on $k$.
In fact, our lower bound holds even for trees of pathwidth $k+1$.
We begin by giving two structural lemmas that allow us to decompose a tree of pathwidth $\ell$ into a path
and a collection of trees of pathwidth at most $\ell-1$.

%\begin{lemma}
%Let $G=(V,E)$ be a graph of pathwidth $k$, for some $k\geq 1$, and let $u\in V(G)$.
%Let $H=(V,E)$ be a graph with
%$$V(H)=v\cup (\{1,2,3\}\times V(G)),$$
%and
%$$E(H)= \bigcup_{i=1}^3 \left(\{v,\{i,u\}\} \cup \bigcup_{\{x,y\}\in E(G)}\{\{i,x\},\{i,y\}\}\right).$$
%Then, the pathwidth of $H$ is $k+1$.
%\end{lemma}
%\begin{proof}
%XXX
%\end{proof}

\begin{lemma}\label{lemma:connect_three_pwd}
Let $G_1,G_2,G_3$ be connected graphs of pathwidth $k$ with disjoint vertex sets, and for $i\in [3]$, let $v_i\in V(G_i)$.
Let $G$ be the graph obtained by introducing a new vertex $v^*$, and connecting it to $v_1$, $v_2$, and $v_3$.
Formally, $V(G) = \{v^*\}\cup \bigcup_{i=1}^3 V(G_i)$, and $E(G)=\bigcup_{i=1}^3 E(G_i)\cup \{\{v^*,v_i\}\}$.
Then $G$ has pathwidth $k+1$.
\end{lemma}

\begin{proof}
It is easy to see that $G$ has pathwidth at most $k+1$:
For each $i\in [3]$ take a path decomposition of $G_i$ with bags $C_{i,1},\ldots,C_{i,{\ell}_i}$.
For each $i\in [3]$, $j\in [\ell_i]$, let $C'_{i,j}=C_{i,j}\cup \{v^*\}$.
The bags $C'_{1,1},\ldots,C'_{1,\ell_1},C'_{2,1},\ldots,C'_{2,\ell_2},C'_{3,1},\ldots,C'_{3,\ell_3}$ induce a path decomposition of $G$ with width at most $k+1$.

Assume now for the sake of contradiction that the pathwidth of $G$ is at most $k$.
That is, there exists a path decomposition of $G$ with bags $C_1,\ldots,C_{\ell}$, such that: (i) for each $i\in [\ell]$, $|C_i|\leq k+1$, (ii) for each $\{u,v\}\in E(G)$ there exists $i\in [\ell]$ with $u,v\in C_i$, and (iii) for each $v\in V(G)$ there exists a subinterval $I\subseteq [\ell]$ such that $v\in C_i$ iff $i\in I$.
For each $i\in [3]$, let $G_i'$ be the subgraph of $G$ induced by $V(G_i)\cup \{v*\}$.
Let also
\[
A_i = \{j\in [\ell] : C_j \cap V(G_i') \neq \emptyset\}.
\]
Note that since $G_i'$ is connected, it follows that $A_i$ is a subinterval of $[\ell]$.
Pick $i_1,i_2\in [3]$, such that $1\in A_{i_1}$, and $\ell\in A_{i_2}$.
Note that we might have $i_1=i_2$.
Since $V(G_{i_1}')\cap V(G_{i_2}') \neq \emptyset$, we have that $A_{i_1}\cap A_{i_2}\neq \emptyset$.
In particular, $A_{i_1}\cup A_{i_2}=[\ell]$.
Therefore, each bag $C_i$ contains at least one vertex either from $G_{i_1}'$, or $G_{i_2}'$.
Let $i_3$ be an element in $[3]\setminus \{i_1,i_2\}$.
Removing $V(G_{i_1}')\cup V(G_{i_2}')$ from all the bags $C_i$, we get a decomposition of $G\setminus (G_{i_1}' \cup G_{i_2}') = G_{i_3}$ with width at most $k-1$, a contradiction since $G_{i_3}$ has pathwidth $k$.
\end{proof}

The following lemma is straightforward.

\begin{lemma}
\label{lem:pwminor}
If $H$ is a minor of $G$, then the pathwidth of $H$ is at most the pathwidth of $G$.
\end{lemma}

\begin{lemma}\label{lemma:pwd_path}
Let $T$ be a tree of pathwidth $\ell \geq 2$.
Then, there exists a simple path $P$ in $T$
such that deleting the vertices of $P$ from $T$ leaves a forest with each tree having pathwidth at most $\ell-1$.
\end{lemma}
\begin{proof}
For every $v\in V(T)$, let $\alpha(v)$ denote the number of connected components of $T\setminus \{v\}$ of pathwidth $\ell$.
We first argue that for any $v\in V(T)$, we have $\alpha(v)\leq 2$.
To see that, assume for the sake of contradiction that there exists $v\in V(T)$, such that $T\setminus \{v\}$ contains connected components $C_1,C_2,C_3$, each of pathwidth at least $\ell$.
Then, by Lemma \ref{lemma:connect_three_pwd} it follows that $T$ must have pathwidth $\ell+1$, a contradiction.

First, observe that if there exists $v\in V(T)$ with $\alpha(v)=0$, then the path contaning only $v$ satisfies the assertion.
Next, we consider the case where for every $v\in V(T)$, $\alpha(v)=1$.
We construct a path $Q=x_1,\ldots,x_s$ as follows.
We set $x_1$ to be an arbitrary leaf of $T$.
Given $x_i$, let $y_i$ be the unique neighbor of $x_i$ in $T$, such that $y_i$ is contained in the unique connected component of $T\setminus \{x_i\}$ of pathwidth $\ell$.
If there exists $j<i$, such that $x_j=y_i$, then we terminate the path $Q$ at $x_i$, and we set $s=i$.
Otherwise, we set $x_{i+1}=y_i$, and continue at $x_{i+1}$.
We now argue that $Q$ satisfies the assertion.
For the sake of contradiction suppose that $T\setminus V(Q)$ contains a connected component $C$ of pathwidth $\ell$.
The component $C$ must be attached to $Q$ via some edge $\{y,x_r\}$, with $y\in V(C)$.
This implies however that $y$ is chosen as $y_r$ when examining $x_r$, and therefore $y$ must be in $Q$, a contradiction.

Finally, it remains to consider the case where there exists at least one $v\in V(T)$, with $\alpha(v)=2$.
Let $X=\{v\in V(T) : \alpha(v)=2\}$.
Let $H=T[X]$ be the subgraph of $T$ induced on $X$.
We first argue that $H$ is connected.
To see this, let $x,y\in X$, and let $L$ be the unique path between $x$ and $y$ in $T$.
Since $\alpha(x)=\alpha(y)=2$, it follows that there exist connected components $C_x,C_y$ of $T\setminus V(L)$ with $C_x$ attached to $x$, and $C_y$ attached to $y$, such that both $C_x$ and $C_y$ have pathwidth $\ell$.
Let $z\in V(L)$.
It follows that there exist components $C_x', C_y'$ of $T\setminus \{z\}$, such that $C_x\subseteq C_x'$, and $C_y\subseteq C_y'$, which implies that $\alpha(z)=2$.
Thus, $z\in X$.
This implies that $L\subseteq H$, and therefore $H$ must be connected.

We next show that $H$ is a path.
To see this suppose for the sake of contradiction that there exists $v\in V(H)$ with distinct
neighbors $v_1,v_2,v_3\in V(H)$.
Since $\alpha(v_1)=\alpha(v_2)=\alpha(v_3)=2$, it follows that there exist components $C_1,C_2,C_3$ of $T\setminus \{v,v_1,v_2,v_3\}$, with each $C_i$ adjacent to $v_i$, and such that each $C_i$ has pathwidth $\ell$, for all $i\in \{1,2,3\}$.
By Lemma \ref{lem:pwminor} we have that for any $i\in \{1,2,3\}$, the connected component of $T\setminus \{v\}$ containing $v_i$ has pathwidth at least $\ell$.
Applying Lemma \ref{lemma:connect_three_pwd}, we obtain that $T$ has pathwidth at least $\ell+1$, a contradiction.
Therefore, $H$ is a path.

Let $w_1,w_2$ be the two endpoints of the path $H$.
We remark that we might have $w_1=w_2$, if there is only one vertex in $H$.
Since $\alpha(w_1)=2$, it follows that there exists a connected component of  $C_{w_1}$ of $T \setminus V(H)$ of pathwidth $\ell$ which is attached to $w_1$.
Similarly, there exists a connected component $C_{w_2}$ of $T \setminus V(H)$ of pathwidth $\ell$ which is attached to $w_2$.
Note that even if $w_1=w_2$, since $\alpha(w_1)=2$, the components $C_{w_1}$, $C_{w_2}$ can be chosen to be distinct.
Let $w_1',w_2'$ be the neighbors of $w_1$, and $w_2$ in $C_{w_1}$, and $C_{w_2}$ respectively.
Let $H'$ be the path obtained by adding $w_1'$, and $w_2'$ to $H$.

We will show that $Q=H'$ satisfies the assertion of the lemma.
To that end, it remains to show that any connected component of $T\setminus V(H')$ has pathwidth at most $\ell-1$.
Let $C$ be a component of $T\setminus V(H')$, and suppose for the sake of contradiction that it has pathwidth $\ell$.
Suppose first that $C$ is attached to a vertex $v\in V(H)$.
Since $v\in V(H)$, it follows that $\alpha(v)=2$.
By applying Lemma \ref{lemma:connect_three_pwd} on the clusters $C$, $C_{w_1}$, and $C_{w_2}$, we obtain that $T$ contains a minor of pathwidth at least $\ell+1$, which combined with Lemma \ref{lem:pwminor} leads to a contradiction.

Finally, suppose that $C$ is attached to a vertex $w\in \{w_1',w_2'\}$, and assume, without
loss of generality, that $w=w_1'$.
Then it follows that $T\setminus \{w\}$ contains at least two components of pathwidth $\ell$ (one containing $C$, and another containing $C_{w_2}$), and thus $\alpha(w)=2$, a contradiction since $w\notin X$.
This concludes the proof.
\end{proof}

We now state the main result of this Section.

\begin{theorem}\label{thm:nonembed_quantitative}
For any $k\geq 1$, and for any $n\geq 1$, there exists an $n$-vertex tree $G$ of pathwidth $k+1$, such that any stochastic
$D$-embedding of $G$ into metric graphs of pathwidth $k$, has $D \geq \Omega(n^{2^{-k}})$.
In particular, $\mathsf{Pathwidth}(k+1) \cap \mathsf{Trees} \not\embeds \mathsf{Pathwidth}(k)$.
\end{theorem}

The remainder of this Section is devoted to proving Theorem \ref{thm:nonembed_quantitative}.
We first construct a graph that will be used for the lower bound.
For each $i\geq 1$, let $\Phi_{i}$ be the unit-weighted graph consisting of a vertex $v$ connected
 to $i$ disjoint paths of length $i$.
Observe that $\Phi_i$ is a tree with $i$ leaves.
We consider $\Phi_i$ as being rooted at the vertex $v$.

For each $i\geq 1$, and for each $m\geq 1$, we define the graph $\Psi_{i,m}$ as follows.
For $i=1$, we set $\Psi_{1,m}=\Phi_{\lceil \sqrt{m}\rceil}$.
For $i\geq 2$, let $\Psi_{i,m}$ be the graph obtained by identifying the root of a copy of $\Psi_{i-1,\sqrt{m}}$, with each leaf of $\Phi_{\lceil \sqrt{m}\rceil}$.
For $\ell \leq \lceil \sqrt{m}\rceil$, let $\Psi_{i,m,\ell}$ be the tree obtained from $\Psi_{i,m}$ by deleting $\lceil\sqrt{m}\rceil-\lceil \ell\rceil$ children
of the root of $\Psi_{i,m}$, along with everything underneath those children.
In particular, $\Psi_{i,m,\sqrt{m}} = \Psi_{i,m}$.

\begin{figure}
\begin{center}
\includegraphics[width=3in]{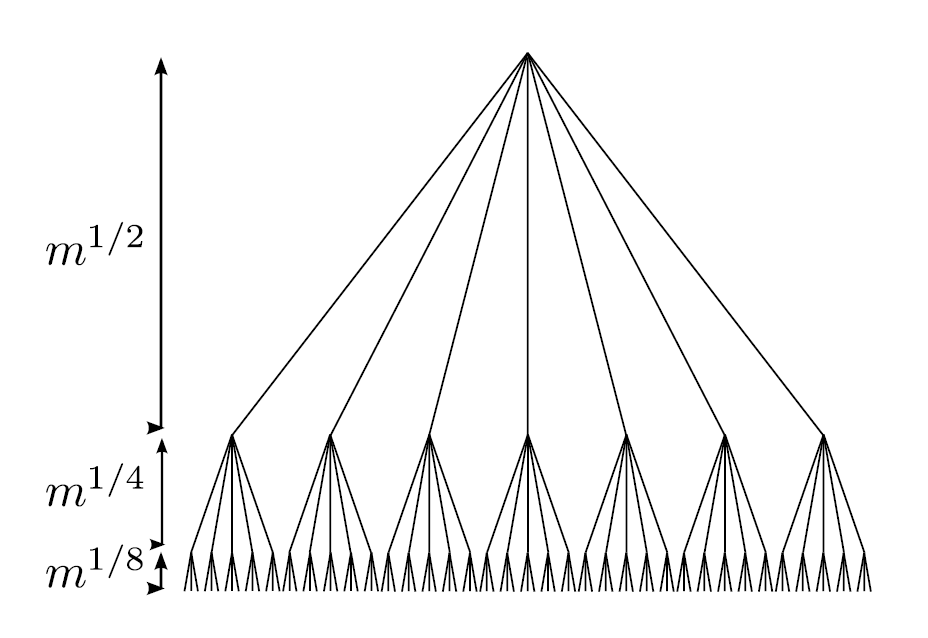}
\caption{The tree $\Psi_{3,m}$.\label{fig:lower}}
\end{center}
\end{figure}

\begin{lemma}\label{lemma:Psi_pwd}
For each $i\geq 1$ and $m\geq 3^{2^i}$, $\Psi_{i,m}$ has pathwidth $i+1$.
\end{lemma}
\begin{proof}
Note that for $m \geq 3^{2^i}$, $\Psi_{i,m}$ contains, as a minor, a full ternary tree $T$ of depth $i$.
Using Lemma \ref{lemma:connect_three_pwd} inductively shows that the pathwidth of the depth $i$
ternary tree is $i+1$, hence Lemma \ref{lem:pwminor} implies that the pathwidth of $\Psi_{i,m}$
is at least $i+1$.  It is also easy to check by hand that the pathwidth of $\Psi_{i,m}$ is at most $i+1$,
for every $m \geq 1$.
\end{proof}

Fix $k\geq 1$, and let $G=\Psi_{k,m}$.
By Lemma \ref{lemma:Psi_pwd}, for $m \geq 3^{2^k}$, $G$ has pathwidth $k+1$.
We will show that for $m$ large enough,
any stochastic $c$-embedding of $(V(G),d_G)$ into a distribution over metric graphs of pathwidth $k$, has distortion $c \geq \Omega(n^{2^{-k}})$,
were $n = |V(G)|$.

\medskip

Assume there exists a stochastic $c$-embedding of $(V(G),d_G)$ into a distribution over metric graphs of pathwidth $k$.
By composing this with the result of Theorem \ref{thm:pathwidth-into-trees}, we get a stochastic $c'$-embedding
of $(V(G),d_G)$ into the family of metrics supported on $\mathsf{Pathwidth}(k) \cap \mathsf{Trees}$, with $c' = O(c)$ (where the $O(\cdot)$
notation hides a constant depending on the fixed parameter $k$).

By averaging, there exists a metric tree $T$ of pathwidth $k$ and a single non-contractive mapping $f : V(G) \to V(T)$
which satisfies,
\[
\frac{1}{|E(G)|} \sum_{\{u,v\}\in E(G)} d_T(f(u),f(v))\leq c'.
\]
Thus it suffices to prove a lower bound on this quantity.
In fact we will prove a somewhat stronger statement; we will give a lower bound on the average stretch of any non-contractive embedding of $\Psi_{k,m,\sqrt{m}/2}$.
We first prove an auxiliary lemma.

\begin{lemma}\label{lem:close_to_P}
Let $S$ be an unweighted tree with $r\in V(S)$, and let $L \geq 0$.
Let $\ell \geq 0$, and let $S_1,\ldots,S_{\ell}$ be vertex-disjoint subtrees of $S$ such that for each $i\in [\ell]$, $S_i$ is attached to $r$ via a path $Q_i$ of length at least $L$, and
for each $i\neq j\in [\ell]$, the paths $Q_i$ and $Q_j$ intersect only at $r$.
We remark that each $S_i$ might contain only a single vertex.
Let $g : V(S) \to V(T)$ be a non-contractive embedding of $S$ into a metric tree $T$, and let $P$ be a simple path in $T$.
If $I=\{i\in [\ell]:d_T(g(S_i),P) < L/2\}$, then
\[
\sum_{\{u,v\}\in E(S)} d_T(g(u),g(v)) \geq
\frac{|I|^2 L}{16}.
\]
\end{lemma}

\begin{proof}
For each $i\in I$, let
$z_i = \mbox{argmin}_{v\in P}\, d_T(v, g(S_i))$,
and
\[
B_i = \{x\in V(P) : d_T(z_i,x)\leq L/2\}.
\]
Since $g$ is non-contractive, we have that for each $i,j \in I$ with $i\neq j$, $d_T(g(S_i), g(S_j))\geq 2L$, and therefore
$$d_T(z_i,z_j)\geq d_T(g(S_i),g(S_j))-d_T(g(S_i),z_i)-d_T(g(S_j),z_j) = 2L-d_T(S_i,P)-d_T(S_j,P) > L,$$ which implies $B_i\cap B_j=\emptyset$.

By reordering, we assume that $I=\{1,2,\ldots,|I|\}$, and that for each $i,j \in I$ with $i<j$, $B_i$ appears to the left of $B_j$ in $P$, after fixing some orientation of $P$.
Furthermore, by choosing the proper orientation, we may assume that there
is a vertex $u_0 \in P$ such that $u_0$ is contained in, or appears to the left of $B_{\lceil |I|/2\rceil}$ in $P$,
and $g(r)$ and $u_0$ are in the same subtree of $T \setminus E(P)$.

For each $i\in \{1,\ldots,|I|\}$, let $w_i = \mbox{argmin}_{v\in V(S_i)} d_S(v,r)$.
It follows that for each $i\in \{\lceil |I|/2\rceil+1,\ldots,|I|\}$, $$d_T(g(r),g(w_i)) \geq \left(i-\lceil |I|/2\rceil\right) L.$$
Furthermore, for every $i \in [\ell]$, clearly $d_T(g(r), g(w_i)) \geq L$ by non-contractiveness.
Therefore,
\begin{align*}
\sum_{\{u,v\}\in E(S)} d_T(g(u),g(v)) &\geq \sum_{i\in \{\lceil 1,\ldots,|I|\}} \sum_{\{u,v\}\in E(Q_i)} d_T(g(u), g(v))\\
  &\geq \sum_{i\in \{\lceil 1,\ldots,|I|\}} d_T(g(r), g(w_i))\\
  &\geq  \left\lceil \frac{|I|}{2}\right\rceil L +  \sum_{i\in \{\lceil |I|/2\rceil+1,\ldots,|I|\}} \left(i-\left\lceil\frac{|I|}{2}\right\rceil\right) L
  >  \frac{|I|^2 L}{16}.
\end{align*}
\end{proof}

The proof of the lower bound proceeds by induction on $k$.
We first prove the base case for embedding into trees of pathwidth one.

\begin{lemma}\label{lem:lower-base-case}
Let $g : V(\Psi_{1,m,\sqrt{m}/2}) \to V(T)$ be a non-contracting embedding into a metric tree $T$ of pathwidth one.
Then,
$$
\frac{1}{|E(\Psi_{1,m,\sqrt{m}/2})|} \sum_{\{u,v\} \in E(\Psi_{1,m,\sqrt{m}/2})} d_T(g(u),g(v)) \geq \frac{\sqrt{m}}{2^{10}}.
$$
\end{lemma}

\begin{proof}
Since the tree $T$ has pathwidth one, it consists of a path $P=\{ v_1,\ldots,v_t \}$, and a collection of vertex-disjoint stars $T_1,\ldots,T_t$, with each $T_i$ being rooted at $v_i$.
Note that $T_i$ might contain only the vertex $v_i$.

Recall that $\Psi_{1,m,\sqrt{m}/2}$ consists of $\lfloor \sqrt{m}/2\rfloor$ disjoint paths $Q_1,\ldots,Q_{\lfloor \sqrt{m}/2\rfloor}$, with $$Q_i=\{r,q_{i,1},\ldots,q_{i,\lceil \sqrt{m}\rceil}\},$$ where $r$ is the root of $\Psi_{1,m,\sqrt{m}/2}$.
For each $i\in [\lfloor \sqrt{m}/2\rfloor]$ let $Q_i'$ be the subpath of $Q_i$ of length $\lfloor \sqrt{m}/2\rfloor $ with $Q_i'=\left\{q_{i,\lceil \sqrt{m}/2\rceil},\ldots,q_{i,\lceil\sqrt{m}\rceil}\right\}$.

Let
$I_1 = \{i\in [\lfloor \sqrt{m}/2\rfloor] : d_T(g(Q_i'), P) \geq \sqrt{m}/4\}$,
and let $I_2=[\lfloor \sqrt{m}/2\rfloor]\setminus I_1$.
By Lemma \ref{lem:close_to_P},
$$
\sum_{\{u,v\} \in E(\Psi_{1,m,\sqrt{m}/2})} d_T(g(u),g(v)) \geq \frac{|I_2|^2 \sqrt{m}}{32}.
$$
Since $|E(\Psi_{1,m})| \leq 2m$, we are done if $|I_2| \geq \sqrt{m}/4$.

\medskip

It remains to consider the case $|I_1| \geq \lfloor \sqrt{m}/4\rfloor$.
Observe that for each $i\in I_1$, all the edges of $Q_i'$ have their endpoints mapped to distinct
leaves of the stars $T_1,\ldots,T_t$, with the edge adjacent to each such leaf having length at least $\sqrt{m}/4$, by
non-contractiveness of $g$.
Therefore, each edge of such a $Q_i'$ is stretched by a factor of $\sqrt{m}/2$ in $T$.
In other words,
\begin{eqnarray*}
\frac{1}{|E(\Psi_{1,m,\sqrt{m}/2})|} \sum_{\{u,v\}\in E(\Psi_{1,m,\sqrt{m}/2})} d_T(g(u),g(v))  & \geq & \frac{1}{2m} \cdot |I_1| \cdot \left\lfloor\frac{\sqrt{m}}{2} \right\rfloor \frac{\sqrt{m}}{2}\\
& = & \frac{\sqrt{m}}{4} \left\lfloor \frac{\sqrt{m}}{4}\right\rfloor \left\lfloor \frac{\sqrt{m}}{2}\right\rfloor \geq \frac{\sqrt{m}}{32},
\end{eqnarray*}
with the latter bound holding for $m \geq 4$.  Observe that the LHS is always at least 1,
yielding the desired result for $m \leq 4$ as well, and completing the proof.
\end{proof}

We are now ready to prove the main inductive step.

%\newpage
\begin{lemma}\label{lem:lower-general-case}
Let $k\geq 1$, $a \in \mathbb N$, $m = (2a)^{2^k}$, and let $g : V(\Psi_{k,m,\sqrt{m}/2}) \to V(T)$ be a non-contractive embedding of $\Psi_{k,m,\sqrt{m}/2}$ into a metric
tree $T$ of pathwidth $k$.
Then, %for all $m = (2a)^{2^k}$, with $a \in \mathbb N$, we have
$$
\frac{1}{|E(\Psi_{k,m,\sqrt{m}/2})|}
\sum_{\{u,v\} \in E(\Psi_{k,m,\sqrt{m}/2})} d_T\left(g(u),g(v)\right) \geq \frac{m^{2^{-k}}}{2^{7+3k}}\,.
$$
\end{lemma}
\begin{proof}
We proceed by induction on $k$.
The base case $k=1$ is given by Lemma \ref{lem:lower-base-case}, so we can assume that $k\geq 2$, and that the assertion is true for $k-1$.

Since the tree $T$ has pathwidth $k\geq 2$, by Lemma \ref{lemma:pwd_path} it follows that it consists of a path $P$, and a collection of trees $T_1,\ldots,T_t$ of pathwidth at most $k-1$, with each $T_i$ being rooted at some vertex $v_i$, and $v_i$ being attached to $P$ via an edge.
Recall that $\Psi_{k,m,\sqrt{m}/2}$ consists of a root $r$ and $\sqrt{m}/2$ subtrees $Q_1,\ldots,Q_{\sqrt{m}/2}$, with each $Q_i$ having a copy $Q_i'$ of $\Psi_{k-1,\sqrt{m}}$ that is connected to $r$ via a path of length $\sqrt{m}$.

\medskip

Let
$I_1 = \{i\in [\sqrt{m}/2] : d_T(g(Q_i'), P) \geq \sqrt{m}/2\}$,
and let $I_2=[\sqrt{m}/2]\setminus I_1$.
By Lemma \ref{lem:close_to_P},
$$
\sum_{\{u,v\} \in E(\Psi_{k,m,\sqrt{m}/2})} d_T(g(u),g(v)) \geq \frac{|I_2|^2 \sqrt{m}}{16}.
$$
Since $|E(\Psi_{k,m})| \leq km$, this yields the desired result
for $|I_2| \geq \sqrt{m}/4$.

\medskip

It remains to consider the case $|I_1|\geq \sqrt{m}/4$.
Let $I_{1,1}$ be the subset of $I_1$ containing all indices $i\in I_1$ such that
for some $j\in [t]$, $T_j$ contains the image of a copy of $\Psi_{k-1,m^{1/2},m^{1/4}/2}$ from $Q_i'$.
Let also $I_{1,2}=I_1\setminus I_{1,1}$.

By the induction hypothesis it follows that for any $i \in I_{1,1}$,
\begin{equation}\label{eq:lower_I11}
\sum_{\{u,v\}\in E(Q_i')} d_T(g(u),g(v)) \geq m^{1/2}\cdot (k-1) \cdot \frac{(m^{1/2})^{2^{1-k}}}{2^{7+3(k-1)}} = m^{1/2} \cdot (k-1) \cdot \frac{m^{2^{-k}}}{2^{4+3k}}. % F(m^{1/2},k-1)
\end{equation}
Consider now $i\in I_{1,2}$.
Let $r_i$ be the root of $Q_i'$ and let $W_{i,1},\ldots,W_{i,m^{1/4}}$ be the copies of $\Psi_{k-1,m^{1/2},1}$ in $Q_i'$, intersecting only at $r_i$.
By the definition of $I_{1,2}$ we have that for any $J\subset [m^{1/4}]$ with $|J|=m^{1/4}/2$, and for any $i'\in [t]$, $\bigcup_{j\in J}g(W_{i,j}) \nsubseteq T_{i'}$.
Assume that the image of $r_i$ is contained in $T_\tau$, for some $\tau\in [t]$.
It follows that there exists $R\subseteq [m^{1/4}]$, with $|R|\geq m^{1/4}/2$, such that for each $j\in R$, the image of $W_{i,j}$ intersects some tree $T_{\sigma_j}$, with $\sigma_j\neq \tau$.
Since $r_i\in V(W_{i,j})$ it follows that there exists an edge $e_{i,j}\in E(W_{i,j})\cup E(Z_{i,j})$ that is stretched by a factor of at least $m^{1/2}$.
It follows that for any $i \in I_{1,2}$,
\begin{equation}\label{eq:lower_I12}
\sum_{\{u,v\}\in E(Q_i')} d_T(g(u),g(v)) \geq \frac{m^{1/4}}{2} \cdot m^{1/2}.
\end{equation}
By \eqref{eq:lower_I11} we get a lower bound for the average stretch of the edges of every $Q_i'$, with $i\in I_{1,1}$.
Similarly, by \eqref{eq:lower_I12} we get a lower bound for the average stretch of the edges of every $Q_i'$, with $i\in I_{1,2}$.
Thus, combining (\ref{eq:lower_I11}) and (\ref{eq:lower_I12}) we get
\begin{eqnarray*}
\frac{1}{|E(\Psi_{k,m,\sqrt{m}/2})|} \sum_{\{u,v\}\in E(\Psi_{k,m,\sqrt{m}/2})} d_T(g(u),g(v)) & \geq & \frac{1}{k\cdot m} \cdot |I_1| \cdot m^{1/2} \cdot (k-1) \cdot \frac{m^{2^{-k}}}{2^{4+3k}}\\
 & > & \frac{m^{2^{-k}}}{2^{7+3k}},
\end{eqnarray*}
as desired.
\end{proof}

This concludes the proof of Theorem \ref{thm:nonembed_quantitative}.

%We have proved the following theorem.

%\begin{theorem}\label{thm:nonembed_quantitative}
%For any $k\geq 1$, and for any $n\geq 1$, there exists an $n$-vertex tree $G$ of pathwidth $k+1$, such that any stochastic $D$-embedding of $G$ into metric graphs of pathwidth $k$, has $D \geq \Omega(n^{2^{-k}})$. In particular, $\mathsf{Pathwidth}(k+1) \cap \mathsf{Trees} \not\embeds \mathsf{Pathwidth}(k)$.
%\end{theorem}

\subsection*{Acknowledgements}

We thank Andrea Francke and Alexander Jaffe for a careful reading of our arguments,
and numerous valuable suggestions.  We are also grateful to the anonymous
referees for many insightful comments.

\bibliographystyle{alpha}
\bibliography{pathwidth,trees,treewidth}

\newcommand{\etalchar}[1]{$^{#1}$}
\def\cprime{$'$} \def\cprime{$'$}
\begin{thebibliography}{CGN{\etalchar{+}}06}

\bibitem[AR98]{AR98}
Yonatan Aumann and Yuval Rabani.
\newblock An {$O(\log k)$} approximate min-cut max-flow theorem and
  approximation algorithm.
\newblock {\em SIAM J. Comput.}, 27(1):291--301 (electronic), 1998.

\bibitem[Bar96]{Bartal96}
Yair Bartal.
\newblock Probabilistic approximations of metric space and its algorithmic
  application.
\newblock In {\em 37th Annual Symposium on Foundations of Computer Science},
  pages 183--193, October 1996.

\bibitem[Bar98]{Bartal98}
Yair Bartal.
\newblock On approximating arbitrary metrics by tree metrics.
\newblock In {\em Proceedings of the 30th Annual ACM Symposium on Theory of
  Computing}, pages 183--193, 1998.

\bibitem[CG04]{CG04}
D.~Carroll and A.~Goel.
\newblock Lower bounds for embedding into distributions over excluded minor
  graph families.
\newblock In {\em Proceedings of the 12th European Symposium on Algorithms},
  2004.

\bibitem[CGN{\etalchar{+}}06]{CGNRS06}
Chandra Chekuri, Anupam Gupta, Ilan Newman, Yuri Rabinovich, and Alistair
  Sinclair.
\newblock Embedding {$k$}-outerplanar graphs into {$l\sb 1$}.
\newblock {\em SIAM J. Discrete Math.}, 20(1):119--136 (electronic), 2006.

\bibitem[CJLV08]{CJLV08}
Amit Chakrabarti, Alexander Jaffe, James~R. Lee, and Justin Vincent.
\newblock Embeddings of topological graphs: Lossy invariants, linearization,
  and 2-sums.
\newblock In {\em {IEEE} Symposium on Foundations of Computer Science}, 2008.

\bibitem[CKR10]{CKR10}
Eden Chlamtac, Robert Krauthgamer, and Prasad Raghavendra.
\newblock Approximating sparsest cut in graphs of bounded treewidth.
\newblock In {\em APPROX-RANDOM}, pages 124--137, 2010.

\bibitem[CSW10]{CSW10}
Chandra Chekuri, F.~Bruce Shepherd, and Christophe Weibel.
\newblock Flow-cut gaps for integer and fractional multiflows.
\newblock In {\em Proceedings of the 21st Annual ACM-SIAM Symposium on Discrete
  Algorithms}, pages 1198--1208, 2010.

\bibitem[Die05]{DiestelBook}
Reinhard Diestel.
\newblock {\em Graph theory}, volume 173 of {\em Graduate Texts in
  Mathematics}.
\newblock Springer-Verlag, Berlin, third edition, 2005.

\bibitem[FF56]{FF56}
L.~R. Ford and D.~R. Fulkerson.
\newblock Maximal flow through a network.
\newblock {\em Canadian Journal of Mathematics}, 8:399--404, 1956.

\bibitem[FRT04]{FRT04}
Jittat Fakcharoenphol, Satish Rao, and Kunal Talwar.
\newblock A tight bound on approximating arbitrary metrics by tree metrics.
\newblock {\em J. Comput. Syst. Sci.}, 69(3):485--497, 2004.

\bibitem[GNRS04]{GNRS99}
Anupam Gupta, Ilan Newman, Yuri Rabinovich, and Alistair Sinclair.
\newblock Cuts, trees and {$l\sb 1$}-embeddings of graphs.
\newblock {\em Combinatorica}, 24(2):233--269, 2004.

\bibitem[IS07]{IS07}
Piotr Indyk and Anastasios Sidiropoulos.
\newblock Probabilistic embeddings of bounded genus graphs into planar graphs.
\newblock In {\em Proceedings of the 23rd Annual Symposium on Computational
  Geometry}. ACM, 2007.

\bibitem[Kar89]{Karp89}
R.~M. Karp.
\newblock A $2k$-competitive algorithm for the circle.
\newblock Manuscript, 1989.

\bibitem[LLR95]{LLR95}
N.~Linial, E.~London, and Y.~Rabinovich.
\newblock The geometry of graphs and some of its algorithmic applications.
\newblock {\em Combinatorica}, 15(2):215--245, 1995.

\bibitem[Lov06]{Lov06}
L{\'a}szl{\'o} Lov{\'a}sz.
\newblock Graph minor theory.
\newblock {\em Bull. Amer. Math. Soc. (N.S.)}, 43(1):75--86 (electronic), 2006.

\bibitem[LR99]{LR99}
Tom Leighton and Satish Rao.
\newblock Multicommodity max-flow min-cut theorems and their use in designing
  approximation algorithms.
\newblock {\em J. ACM}, 46(6):787--832, 1999.

\bibitem[LR10]{LR10}
James~R. Lee and Prasad Raghavendra.
\newblock Coarse differentiation and multi-flows in planar graphs.
\newblock {\em Discrete Comput. Geom.}, 43(2):346--362, 2010.

\bibitem[LS09]{LS-STOC09}
J.~R. Lee and A.~Sidiropoulos.
\newblock On the geometry of graphs with a forbidden minor.
\newblock In {\em 41st Annual Symposium on the Theory of Computing}, 2009.

\bibitem[OS81]{OS81}
Haruko Okamura and P.~D. Seymour.
\newblock Multicommodity flows in planar graphs.
\newblock {\em J. Combin. Theory Ser. B}, 31(1):75--81, 1981.

\bibitem[RR98]{RR98}
Y.~Rabinovich and R.~Raz.
\newblock Lower bounds on the distortion of embedding finite metric spaces in
  graphs.
\newblock {\em Discrete Comput. Geom.}, 19(1):79--94, 1998.

\bibitem[RS83]{RS83}
Neil Robertson and P.~D. Seymour.
\newblock Graph minors. {I}. {E}xcluding a forest.
\newblock {\em J. Combin. Theory Ser. B}, 35(1):39--61, 1983.

\end{thebibliography}

\end{document}